\documentclass[12pt]{article}
\usepackage[utf8]{inputenc}
\usepackage{graphicx}
\usepackage{amsmath}
\usepackage{fullpage}
\usepackage{amsfonts}
\usepackage{amssymb}
\usepackage{amsthm}
\newtheorem{thm}{Theorem}[section]
\newtheorem{lem}{Lemma}[section]

\newtheorem{ex}{Example}[section]

\newtheorem{rmk}{Remark}[section]
\newtheorem{cnj}{Conjecture}[section]

\title{Spectra of eccentricity matrices of graphs}

\author{Iswar Mahato\thanks{Department of Mathematics, Indian Institute of Technology Kharagpur, Kharagpur 721 302, India. Email: iswarmahato02@gmail.com}\and R. Gurusamy\thanks{Department of Mathematics, Mepco Schlenk Engineering College, Sivakasi 626 005, Tamil Nadu, India. Email: sahama2010@gmail.com}  \and M. Rajesh Kannan\thanks{Department of Mathematics, Indian Institute of Technology Kharagpur, Kharagpur 721 302, India. Email: rajeshkannan@maths.iitkgp.ac.in, rajeshkannan1.m@gmail.com } \and S. Arockiaraj\thanks{Department of Mathematics, Government Arts and Science College, Sivakasi 626 124, Tamil Nadu, India. Email: psarockiaraj@gmail.com}
}
\date{\today}
\begin{document}
\maketitle
\baselineskip=0.25in

\begin{abstract}
The eccentricity matrix of a connected graph $G$  is obtained from the distance matrix of $G$ by retaining  the largest distances in each row and each column, and setting the remaining entries as $0$.  In this article, a conjecture about  the least eigenvalue of eccentricity matrices  of trees, presented in the  article  [Jianfeng Wang, Mei Lu, Francesco Belardo, Milan Randic. The anti-adjacency matrix of a graph: Eccentricity matrix. Discrete Applied Mathematics, 251: 299-309, 2018.],  is solved affirmatively. In addition to this, the spectra and the inertia of eccentricity matrices of various classes of graphs are investigated.
\end{abstract}



\section{Introduction}\label{sec1}

Let $G=(V(G),E(G))$ be a finite simple graph with vertex set $V(G)=\{v_1,v_2,\dots,v_n\}$ and edge set $E(G)=\{e_1, \dots , e_m \}$.  If two vertices $v_i$ and $v_j$ are adjacent, we write $v_i\sim v_j$,   and the edge between them is denoted by $e_{ij}$. The degree of the vertex $v_i$ is denoted by $d_i$. The  \textit{adjacency matrix} of $G$ is an $n \times n$ matrix, denoted by $A(G)$, whose rows and columns are indexed by the vertex set of the graph and the entries are defined by
$$A(G)_{ij}=\begin{cases}
1 &\text{if }v_i\sim v_j,\\
0 &\text{otherwise.}\end{cases}$$ The \emph{distance} between the vertices $v_i, v_j \in V(G)$, denoted $d(v_i,v_j)$, is defined to be the smallest value among the lengths (i.e., the number of edges) of the paths between the vertices $v_i$  and $v_j$.   The \emph{distance matrix} of a connected graph $G$, denoted by $D(G)$, or simply $D$, is the $n\times n$ matrix whose $(i,j)^{th}$-entry  is equal to $d(v_i,v_j)$, $i=1,2,\hdots,n$; $j=1,2,\hdots,n$.  The adjacency matrix and the distance matrix of a graph are  well studied matrix classes in the field of spectral graph theory. For more details about the study of classes of matrices associated with graphs, we refer to \cite{bapatbook, Brou, Cve1}. The \emph{eccentricity} $e(v_i)$ of the vertex $v_i$ is defined  as $e(v_i)=max \{d(v_i,v_j):  v_j \in V(G)\}$.  The \emph{eccentricity matrix of  a connected graph $G$}, denoted by $\epsilon(G)$, is defined as follows:
$${\epsilon(G)}_{ij}=
\begin{cases}
\text{$d(v_i,v_j)$} & \quad\text{if $d(v_i,v_j)=min\{e(v_i),e(v_j)\}$,}\\
\text{0} & \quad\text{otherwise.}
\end{cases}$$

The notion of eccentricity matrix was introduced and studied in \cite{ran1, ecc-main}. The eccentricity matrix is also known as $D_{\max}$-matrix in the literature\cite{ran1, ran2}. One of the main applications of eccentricity matrix is in the field of chemistry.  Also these matrices can be used in testing graphs for isomorphism. For more details about the applications of eccentricity matrices, we refer to \cite{ran1, ran2}.

Eccentricity matrix of a graph is opposite to its adjacency matrix in the following sense.  Eccentricity matrix is obtained from the distance matrix by retaining only the largest distances in each row. Adjacency matrix is obtained from the distance matrix by retaining only the smallest non-zero distances in each row  \cite{ecc-main}. The eccentricity matrix, unlike the adjacency matrix and the distance matrix,  of a connected graph need not be irreducible. The eccentricity matrix of a complete bipartite graph is reducible. In \cite{ecc-main}, the authors asked  the following question: For which connected graphs the eccentricity matrix is irreducible? In the same article, the authors  established that the eccentricity matrix of a tree is irreducible.
In this article, we provide an alternate proof for this result (Theorem \ref{alt-pf}).  

Let $\epsilon_1, \epsilon_2, \hdots \epsilon_n$ denote the eigenvalues of the matrix $\epsilon(G)$. Since, the matrix $\epsilon(G)$ is symmetric, all the $\epsilon$-eigenvalues of $G$ are real. If $\epsilon_1>\epsilon_2>\hdots >\epsilon_k$ be all the distinct $\epsilon$-eigenvalues of $G$, then the $\epsilon$-spectrum of $G$, denoted by $spec_{\epsilon}(G)$,  is denoted by
\[ spec_{\epsilon}(G)=
  \left\{ {\begin{array}{cccc}
   \epsilon_1 &\epsilon_2  &\hdots &\epsilon_k \\
   m_1& m_2& \hdots &m_k\\
  \end{array} } \right\},
  \]
where $m_i$ is the algebraic multiplicity of the eigenvalue $\epsilon_i$ for $1\leq i\leq k$. In \cite{ecc-main}, the authors made the following conjecture:
\begin{cnj}\cite[Conjecture 2]{ecc-main}
 Let $T$ be a tree on $n$ vertices, with $n \geq 3$. Then, $\epsilon_n(T) \leq -2$, and  equality holds   if and only if $T$ is the star.
\end{cnj}

In Theorem \ref{sol-conj}, we provide an affirmative answer to this conjecture. In \cite{ecc-main}, the authors computed the $\epsilon$-spectra of some classes of graphs viz., cycles, $r$-regular graphs with diameter $2$, complete product two graphs, and so on. In this article, we compute the $\epsilon$-spectra of corona of a graph and a complete graph. This enables one to  construct infinitely many pairs of non-isomorphic graphs with same $\epsilon$-spectra. Also, we compute the spectra and the inertia of eccentricity matrices associated with various other classes of graphs.

This  article is organized as follows: In section \ref{sec2}, we collect the definitions and preliminary results needed. In section \ref{sec3}, we establish a proof of the conjecture about the least $\epsilon$-eigenvalue of a tree. Also, we provide an alternate proof of the fact that the eccentricity matrix of a tree is irreducible. In section \ref{sec4}, we compute $\epsilon$-spectra of various classes of graphs. In section \ref{sec5}, we obtain the inertia of eccentricity matrices associated with various classes of graphs.

\section{Definitions, notation and preliminary results}\label{sec2}
Let ${\mathbb {R}^{n\times n}}$ denote the set of all $n\times n$ matrices with real entries. For $A\in {\mathbb {R}^{n\times n}}$, let $A^T$, $\det (A)$, $\det(\lambda I-A)$  and $\sigma(A)$ denote  transpose, determinant, characteristic polynomial and  spectrum (set of all eigenvalues)  of $A$, respectively. Let $J_{n \times n}$ or simply $J_n$ denotes the $n\times n$ matrix with all entries equal to $1$,  and $I_n$ denotes the $n\times n$ identity matrix.

Let $A$ be an $n\times n$ matrix partitioned as
$   A=
\left[ {\begin{array}{cc}
	A_{11} & A_{12} \\
	A_{21} & A_{22} \\
	\end{array} } \right]$,
where $A_{11}$ and $A_{22}$ are square matrices. If $A_{11}$ is nonsingular, then the Schur complement of $A_{11}$ in $A$ is defined as $A_{22}-A_{21}{A_{11}^{-1}}A_{12}$. For Schur complements,  we have $ \det A= (\det A_{11})\det(A_{22}-A_{21}{A_{11}^{-1}}A_{12})$.
Similarly, if $A_{22}$ is nonsingular, then the Schur complement of $A_{22}$ in A is $A_{11}-A_{12}{A_{22}^{-1}}A_{21}$, and we have $\det A= (\det A_{22})\det(A_{11}-A_{12}{A_{22}^{-1}}A_{21}).$

The following result is about the eigenvalues of block matrices.
\begin{lem}\cite{circ-davis}\cite[Lemma 3.9]{ecc-main}\label{b0-b1-lemma}
	Let $ B=
	\left[ {\begin{array}{cc}
		B_0 & B_{1} \\
		B_1 & B_0 \\
		\end{array} } \right]$ be a symmetric $2 \times 2$ block matrix with $B_{0}$ and $B_1$ are square matrices of same order. Then, the spectrum of $B$ is  the union of the spectra of $B_0+B_1$ and $B_0-B_1.$
\end{lem}

The following result will be used in the proof of lemma \ref{lem:A}.
\begin{lem}\cite[Theorem 3.3 ]{ecc-main}\label{lambda}
	Let $B$ be a square matrix of order $n$. If each column sum of $B$ is equal to some eigenvalue (say $\alpha$) of $B$, then $J_{1\times n}{(\lambda I-B)}^{-1}J_{n\times 1}=\frac{n}{\lambda-\alpha}.$
\end{lem}

The following result is about the spectrum of a special kind of block matrix. For the sake completeness we include a proof here.
\begin{lem}\label{lem:A}
	Let $A$ be an $(n+1)\times (n+1)$ matrix of the form \[
	A=
	\left[ {\begin{array}{cc}
		0& 2J_{1\times n} \\
		2J_{n \times 1}& 3J_{n} \\
		\end{array} } \right].
	\]
	Then \[
	\sigma(A)=
	\left\{ {\begin{array}{cc}
		0 & \frac{3n\pm \sqrt {9n^2+16n}}{2}  \\
		(n-1) &1\\
		\end{array} } \right\}.
	\]
\end{lem}
\begin{proof} The characteristic polynomial of A is given by
	\[
	\det(\lambda I_{n+1}-A)=
	\det\left[ {\begin{array}{cc}
		\lambda & -2J_{1\times n} \\
		-2J_{n \times 1}& \lambda I_{n}-3J_{n} \\
		\end{array} } \right].
	\]
	Then,  by Schur complement  formula and lemma \ref{lambda} we have
	\begin{eqnarray*}
		\det(\lambda I_{n+1}-A)
		&=& \det( \lambda I_{n}-3J_{n})\det\big[\lambda-4J_{1\times n}(\lambda I_{n}-3J_{n})^{-1}J_{1\times n}\big]\\
		&=&(\lambda)^{n-1}\big(\lambda-3n\big)\det\Big[\lambda-\frac{4n}{\lambda-3n}\Big]\\
		&=& (\lambda)^{n-1}\big[{\lambda}^2-3n\lambda-4n\big].
	\end{eqnarray*}
	This completes the proof.\end{proof}

If $A$ and $B$ are matrices of order $m \times n $ and $p \times q$, respectively, then the Kronecker product of the matrices $A$ and $B$, denoted $A\otimes B$, is the $mp\times nq$ block matrix $[a_{ij}B]$.

\begin{lem}\cite{hor-john-topics}\label{kron-eigen}
	Let $A\in {\mathbb {R}^{n\times n}}$ and $B\in {\mathbb {R}^{m\times m}}$. If $\sigma(A)=\{\lambda_1,\hdots,\lambda_n\}$ and $\sigma(B)=\{\mu_1,\hdots,\mu_m\}$ are the spectra of $A$ and $B$, respectively, then $\sigma(A\otimes B)=\{\lambda_i\mu_j:i=1,\hdots,n;~ j=1,\hdots,m\}.$
\end{lem}

The following theorem is known as interlacing theorem.
\begin{thm}\cite{hor-john-mat}
	Suppose $A \in{\mathbb {R}^{n\times n}}$ is symmetric. Let $B\in {\mathbb{R}^{m\times m}}$
	with $m < n$ be a principal submatrix of $A$ (submatrix whose rows and columns are indexed by the same index set $\{i_1, \ldots, i_m\}$, for some $m$). Suppose $A$ has eigenvalues $\lambda_1 \leq \hdots \leq \lambda_n$
	and B has eigenvalues $\beta_1 \leq  \hdots \leq  \beta_m$. Then,
	$\lambda_k \leq \beta_k \leq  \lambda_{k+n-m}$ for $k = 1,\hdots , m$,
	and if $ m = n-1$, then 
	$\lambda_1\leq \beta_1 \leq \lambda_2 \leq \beta_2 \leq \hdots \leq \beta_{n-1} \leq \lambda_n$.
\end{thm}

The inertia of  a symmetric matrix $A$ is the triple $\big(n_{+}(A),n_{-}(A),n_{0}(A)\big)$, where $n_{+}(A), n_{-}(A)$ and $n_{0}(A)$ denote the number of positive, negative and zero eigenvalues of $A$, respectively.

An $n \times n$ nonnegative matrix $A $ is said to be \emph{reducible} if there exists an $n \times n$ permutation matrix $Q$ such that $QAQ^T =
\begin{pmatrix}
A_{11} & A_{12}  \\
0  & A_{22}
\end{pmatrix}
$, where $A_{11}$ is a $r \times r$ sub matrix  with $1 \leq r < n$. If no such
permutation matrix $Q$ exists, then $A$ is said to be \emph{irreducible}.
Given an $n \times n$ nonnegative symmetric matrix A, associate a graph $G(A)$ on $n$ vertices, say $\{u_1, \dots , u_n\}$ such that there is an edge between the vertices $u_i$ and $u_j$ in $G(A)$ if and only if the $(i,j)^{th}$ entry of $A$ is nonzero.
An  $n \times n$ symmetric nonnegative matrix  $A$ is  irreducible if and only if $G(A)$ is  connected.

Now, let us collect some of the definitions related to  graphs which will be used later. For graph $G$, the complement of $G$, denoted by $\overline{G}$, is a graph with vertices are same as that of $G$, and two vertices $v_i$ and $v_j$ are adjacent in $\overline{G}$ if and only if $v_i$ and $v_j$ are not adjacent in $G$.    For two vertices $v_i, v_j \in V(G)$, let $P(v_i, v_j)$ denote the path joining the vertices $v_i$ and $v_j$.  Let $K_n$ and $P_n$ denote the complete graph and the path on $n$ vertices, respectively. A graph $G$ is said to be $r$-regular, if degree of each vertex of $G$ is $r$. The star graph on $n$ vertices, denoted by $K_{1,n-1}$,  is the tree with one vertex having degree $n-1$. The wheel graph $W_{n+1}$ on $n+1$ vertices is the graph that contains a cycle $C_n$ of length $n$, and a vertex $v$ not in the cycle such that $v$ is adjacent to every  vertex in $C_n$. The $n$-barbell graph $B_{n,n}$ is the graph obtained by connecting two copies of complete graph $K_n$ by a bridge (cut edge).  The cocktail-party graph $CP(n)$ is $(2n-2)$ regular graph with $2n$ vertices obtained by removing $n$ independent edges from $K_{2n}$. Let $G_1\cup G_2$ denote the disjoint union of graphs $G_1$ and $G_2$. Then the complete product $G_1 \vee G_2$ is the graph obtained from $G_1\cup G_2$ by joining every vertex of $G_1$ with every vertex of $G_2$. The $(m,n)$-lollipop graph is the graph obtained by joining a complete graph $K_m$ with a path graph $P_n$ by a bridge. Usually the $(m,n)$-lollipop graph is denoted by $L_{m,n}$.   Let $G$ and $H$ be two graphs on disjoint vertex sets of order $n$ and $m$, respectively. Let $\{v_1, \ldots, v_n\}$ be the vertex set of $G$. The corona $G\circ H$ of $G$ and $H$ is defined as the graph obtained by taking one copy of $G$ and $n$ disjoint copies of $H$, say $H_1, \ldots, H_n$,  and joining the  vertex $v_i$ of $G$ to every vertex in $H_i$, the $i^{th}$ copy of $H$ \cite{har-coro}.   For more details about graphs, we refer to \cite{bon-mur-book, har-book}.

\section{Proof of the conjecture}\label{sec3}

In this section, first we shall compute the $\epsilon$-spectra of star on $n$-vertices. In Theorem \ref{sol-conj}, we provide a solution to the conjecture \cite[Conjecture 2]{ecc-main}.  Finally, we establish an alternate proof for the fact that, if  $T$ is a tree, then $\epsilon(T)$ is irreducible.

\begin{thm}\label{star-eig}
	Let $K_{1,n-1}$ be the star graph on $n$ vertices. Then,  $ \det(\epsilon(K_{1,n-1}))=(-1)^{n-1}(n-1)2^{n-2}$, and
	$ spec_{\epsilon}(K_{1,n-1})=
	\left\{ {\begin{array}{cc}
		(n-2)\pm \sqrt{n^2-3n+3}& -2 \\
		1 & n-2\\
		\end{array} } \right\}.$ Further, if $n \geq 3$, then the least eigenvalue of $\epsilon(K_{1,n-1})$, $\epsilon_n(K_{1,n-1})$ is -2.
\end{thm}
\begin{proof}
	Let $K_{1,n-1}$ be the star graph on $n$ vertices  $\{v_1,v_2,\hdots,v_n\}$, where $v_1$ is the vertex of degree $(n-1)$. Then,
	\[
	\epsilon(K_{1,n-1})=
	\left[ {\begin{array}{cc}
		0& J_{1\times n-1} \\
		J_{n-1 \times 1}& 2(J_{n-1}-I_{n-1}) \\
		\end{array} } \right].
	\]
	Since $2(J_{n-1}-I_{n-1})$ is an $(n-1)\times (n-1)$ nonsingular matrix,  by Schur complement formula, we have
	\begin{eqnarray*}
		\det \big (\epsilon (K_{1,n-1})\big )
		&=& \det\big[2(J_{n-1}-I_{n-1})\big] \det\Big[0- J_{1\times n-1}\big(2(J_{n-1}-I_{n-1})\big)^{-1}J_{n-1 \times 1}\Big]\\
		&=&(-1)^{n-2}(n-2)2^{n-2}\det\Big[ J_{1\times n-1}\big(I_{n-1}-J_{n-1}\big)^{-1}J_{n-1 \times 1}\Big]\\
		&=&(-1)^{n-2}(n-2)2^{n-2}\det \Big[\frac{n-1}{1-(n-1)}\Big]\\
		&=&(-1)^{n-1}(n-1)2^{n-2}.\\
	\end{eqnarray*}
	Now, the characteristic polynomial of $\epsilon(K_{1,n-1})$ is
	\[\det \big(\epsilon(K_{1,n-1})-\lambda I_n \big)=
	\det \left[ {\begin{array}{cc}
		-\lambda &  J_{1\times n-1} \\
		J_{n-1 \times 1} & 2(J_{n-1}-I_{n-1})-\lambda I_{n-1} \\
		\end{array} } \right]\]
	By Schur complement formula, we have
	\begin{eqnarray*}
		\det \big (\epsilon (K_{1,n-1})-\lambda I_n)\big )&=& (-\lambda) \det\big[2(J_{n-1}-I_{n-1})-\lambda I_{n-1}-J_{n-1 \times 1}{(-\lambda)}^{-1}J_{1 \times n-1}\big]\\
		&=& (-\lambda) \det\big[2(J_{n-1}-I_{n-1})-\lambda I_{n-1}+\frac{1}{\lambda}J_{n-1}\big]\\
		&=&(-\lambda) \det\big[(2+\frac{1}{\lambda})J_{n-1}-(2+\lambda)I_{n-1}\big]\\
		&=&(-\lambda)\big[(n-1)(2+\frac{1}{\lambda})-(2+\lambda)\big](-2-\lambda)^{n-2}\\
		&=&\big({\lambda}^2-(2n-4)\lambda-(n-1)\big)(-2-\lambda)^{n-2}\\
		&=&\Big[\lambda-\Big( (n-2)\pm \sqrt{n^2-3n+3}\Big)\Big](-2-\lambda)^{n-2}.\\
	\end{eqnarray*}
	So, $ spec_{\epsilon}(K_{1,n-1})=
	\left\{ {\begin{array}{cc}
		(n-2)\pm \sqrt{n^2-3n+3}& -2 \\
		1 & n-2\\
		\end{array} } \right\}.$
	
	Now, if $n \geq 3$, then   $-2$ is the least $\epsilon$-eigenvalue of $K_{1, n-1}$. Suppose not, then $(n-2) - \sqrt{n^2-3n+3} < -2$, that is, $n < \sqrt{n^2-3(n-1)}$, which is not possible.
\end{proof}

\begin{thm}\label{sol-conj}
	Let $T$ be a tree of order $n$, other than $P_2$, and $\epsilon_n(T)$ be the least $\epsilon$-eigenvalue of $T$. Then,  $\epsilon_n(T)\leq -2$ with equality if and only if $T$ is the star.
\end{thm}
\begin{proof}
	Let $T$ be a tree on $n$ vertices, where $n \geq 3.$ we shall show that, the inequality $\epsilon_n(T)<-2$  holds for any tree $T$ other than the star. Using this and  Theorem \ref{star-eig}, we can conclude that the conjecture is true.
	
	To show  $\epsilon_n(T)<-2$,  we shall show that always there  exists a $2 \times 2 $ principal submatrix of $\epsilon(T)$ such that one of whose eigenvalues is  less than or equal to  $-3$.
	
	Let $T$ be tree on $n$ vertices, other than the star. Without loss of generality,  let $P(v_1,v_n)$ be the longest path in $T$. Then, $d(v_1,v_n)=k$ and $3\leq k\leq n-1$. Since $P(v_1,v_n)$ is the longest  path and $d(v_1,v_n)=k$, the eccentricities of $v_1$ and $v_n$ are equal to $k$. Hence, the $(1,n)^{th}$ entry  of  $\epsilon(T)$ is $k$.  That is,   \[
	\left[ {\begin{array}{cc}
		0&k \\
		k& 0 \\
		\end{array} } \right],
	\] is a   $2\times 2$ principal submatrix of $\epsilon(T)$.
	Now, the eigenvalues of the matrix $\left[ {\begin{array}{cc}
		0&k \\
		k& 0 \\
		\end{array} } \right]$  are $k,-k$, with $3\leq k \leq n-1$. Therefore, by interlacing theorem, $\epsilon(T)$ must have an eigenvalue less than or equal to $-3$. This completes the proof.
\end{proof}
In the next theorem, we provide an alternate proof for \cite[Theorem 1]{ecc-main}.

\begin{thm}\label{alt-pf}
	The eccentricity matrix of a tree is irreducible.
\end{thm}
\begin{proof} Let $T$ be a tree on $n$ vertices with vertex set $V(T)=\{v_1,v_2,\hdots,v_n\}$, and let $\epsilon(T)$ be the eccentricity matrix of $T$. Let $P(v_1,v_n)$ be a path of longest length, say $t$, in $T$. Then $v_1$ and $v_n$ must be pendant vertices in T. Let us relabel the vertices of the tree $T$ with respect to $P(v_1,v_n)$ as follows:
	Let $v_1=v_{s_1},v_{s_2},\hdots,v_{s_t}=v_n$ be the vertices in $P(v_1,v_n)$.
	Let $B_{s_l}$ denote the collection of all branches at the vertex $v_{s_l}$, other than the branches containing the vertices belong to the path $P(v_1,v_n)$.
	Let $B_{s_1},B_{s_2},\hdots,B_{s_t}$ be the collection of all branches, defined as above, at the vertices $v_{s_1},v_{s_2},\hdots,v_{s_t}$, respectively. Note that, $B_{s_1}$ and $B_{s_t}$ are empty.\\
	
	\hfill\includegraphics[scale= 0.50]{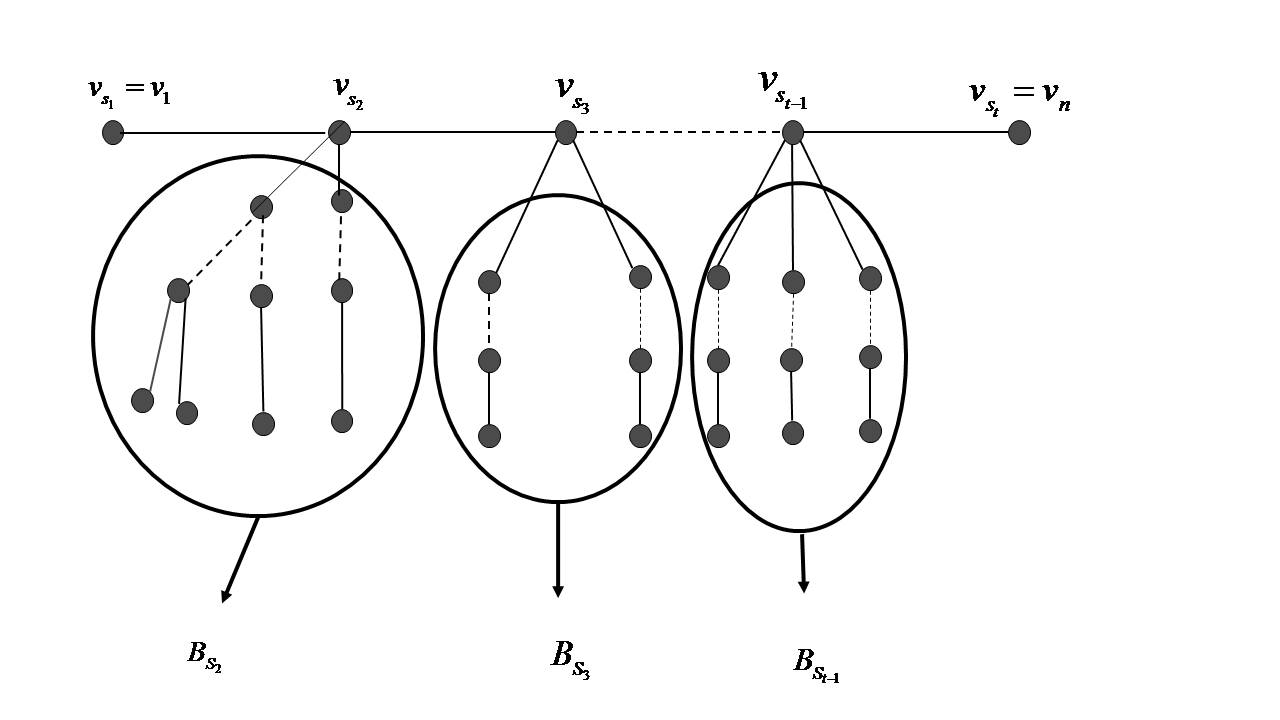}\hspace*{\fill}
	
	To prove $\epsilon(T)$ is irreducible, it is enough to show that the underlying graph $G(\epsilon(T))$ of $\epsilon(T)$ is connected. We shall prove that, for every $v_j\in \{2,3,\hdots,n-1\}$, either $\epsilon(T)_{1j}\neq 0$ or $\epsilon(T)_{nj}\neq 0$, so that $G(\epsilon(T))$ is connected.
	
	Suppose that, for some vertex $v_j$, both  $\epsilon(T)_{1j}= 0$ and $\epsilon(T)_{nj}= 0$. Then there is a vertex $v_k$ in $T$ such that $d(v_j, v_k)>d(v_1,v_j)$ and $d(v_j,v_k)>d(v_n,v_j)$.
	
	\textbf{Case(I):} If $v_j{\in P(v_1,v_n)}$, then either the path $P(v_1,v_k)$ or the path $P(v_n,v_k)$ has length larger than the length of the path $P(v_1, v_n)$ in $T$, which is not possible.
	
	\textbf{Case(II):} Let $v_j$ and $v_k$ be in same $B_{s_i}$, for some $i=2,\hdots,t-1$. If $v_j{\in P(v_1,v_k)}$, then either the path $P(v_1,v_k)$ or the path $P(v_n,v_k)$  have length larger than length of $P(v_1,v_n)$, which contradicts that $P(v_1,v_n)$ has maximum length in T.
	
	Suppose  $v_j\notin P(v_1,v_k)$. Let  $v_l$ be the vertex in $P(v_j,v_k)$ such that $d(v_1,v_l)$ is minimum among all the vertices in the path $P(v_j,v_k)$. Therefore, $d(v_k,v_l)>d(v_1,v_l)$, and hence $d(v_{s_i},v_k)>d(v_{s_i},v_1)$.  Adding $d(v_n,v_{s_i})$ on both the sides in $d(v_{s_i},v_k)>d(v_{s_i},v_1)$, we get $ d(v_n,v_{s_i})+d(v_{s_i},v_k)> d(v_n,v_{s_i})+d(v_{s_i},v_1)$. That is,
	$ d(v_n,v_k)>d(v_n,v_1)$, which is not possible.
	
	\textbf{Case(III):} Let $v_j$ and $v_k$ be in the branches corresponding to different vertices in the path $P(v_1, v_n)$, say, $v_j\in B_{s_h}$ and $v_k\in B_{s_i}$. Without loss of generality, assume that the vertex $v_{s_h}$ is nearer to $v_1$ than the vertex $v_{s_i}$. Since, $d(v_j,v_k)>d(v_j,v_n)$, we have $ d(v_{s_i},v_k)>d(v_{s_i}, v_n)$. Now, adding $d(v_1,v_{s_i})$ on both sides, we get $d(v_1,v_{s_i})+d(v_{s_i},v_k)> d(v_1,v_{s_i})+d(v_{s_i}, v_n)$. That is, $ d(v_1,v_k)>d(v_1,v_n)$,
	which is again a contradiction.
	
	Hence, either $\epsilon(T)_{1j}\neq 0$ or $\epsilon(T)_{nj}\neq 0$, which proves that the underlying graph  $G(\epsilon(T))$ is connected, and hence $\epsilon(T)$ is irreducible.
\end{proof}

\section{$\epsilon$-spectra of some classes of graphs}\label{sec4}

In this section, we compute $\epsilon$-eigenvalues of some classes of graphs.
First, we compute the $\epsilon$-spectrum of corona of any graph $G$ with the complete graph on $n$-vertices. Surprisingly, the spectrum of the corona does not depend on the structure of the graph $G$.

\begin{thm}\label{spec-cor}
	Let $K_n$ be the complete graph on $n$ vertices, and let $G$ be any connected graph on $m$ vertices. Then \[
	spec_{\epsilon}(K_n\circ G)=
	\left\{ {\begin{array}{ccccc}
		0& -\lambda_1&-\lambda_2 &\lambda_1(n-1)& \lambda_2(n-1)  \\
		n(m-1) & n-1 & n-1 & 1& 1\\
		\end{array} } \right\},
	\]
	where $\lambda_1$ and $\lambda_2$ are the roots of $x^2-3mx-4m=0$.
\end{thm}
\begin{proof} Let $K_n$ be the complete graph on $n$ vertices, and let $G$ be any connected graph on $m$ vertices. Then, the graph $K_n\circ G$ consists of  $n$ vertices of the complete graph $K_n$ which are labeled using the index set $\{1,2,\hdots,n\}$, and $n$ disjoint copies $G_1, G_2,\hdots,G_n$ of $G$. Choose an arbitrary ordering $g_1,g_2,\hdots,g_m$ of the vertices of $G$, and label the vertices of   $G_i$ corresponding to $g_k$ by the indices $i+nk$ \cite{cor-spec}.

	Under this labeling, the eccentricity matrix of $K_n\circ G$ is given by $\epsilon (K_n\circ G)=A\otimes B$, where
	$A=
	\left[ {\begin{array}{cc}
		0& 2J_{1\times m} \\
		2J_{m \times 1}& 3J_{m} \\
		\end{array} } \right]$ and $B=J_n-I_n$. \\
	By Lemma \ref{lem:A},  we have
	$ \sigma(A)=
	\left\{ {\begin{array}{cc}
		0 & \frac{3m\pm \sqrt {9m^2+16m}}{2}  \\
		(m-1) &1\\
		\end{array} } \right\}$
	and
	$ \sigma(B)=
	\left\{ {\begin{array}{cc}
		-1 & (n-1) \\
		(n-1) &1\\
		\end{array} } \right\}.$
	Now, by Lemma \ref{kron-eigen}, the spectrum of $A\otimes B$ is
	\[
	\sigma(A\otimes B)=
	\left\{ {\begin{array}{ccccc}
		0& -\lambda_1&-\lambda_2 &\lambda_1(n-1)& \lambda_2(n-1)  \\
		n(m-1) & n-1 & n-1 & 1& 1\\
		\end{array} } \right\},
	\]
	and hence \[
	spec_{\epsilon}(K_n\circ G)=
	\left\{ {\begin{array}{ccccc}
		0& -\lambda_1&-\lambda_2 &\lambda_1(n-1)& \lambda_2(n-1)  \\
		n(m-1) & n-1 & n-1 & 1& 1\\
		\end{array} } \right\}.
	\]
	
\end{proof}

\begin{ex}
	Let us illustrate the labeling process which is used in the above theorem with the following particular example: 
	
	\hfill\includegraphics[scale= 0.50]{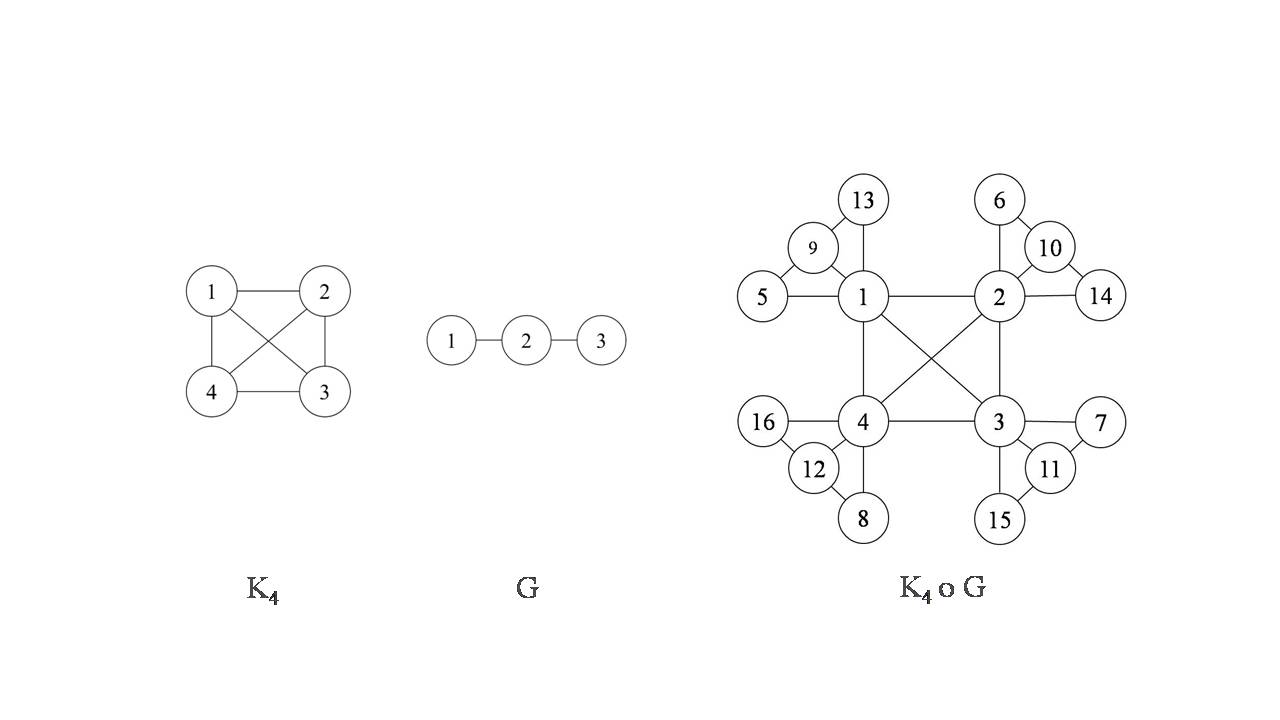}\hspace*{\fill}
	
\end{ex}

\begin{rmk}
	Instead of taking  $n$ copies of the graph $G$, if we take $n$ different connected graphs $G_1, \ldots, G_n$ with $m$ vertices and add the vertex $v_i$ of $K_n$ to every vertex of $G_i$, $i = 1, 2\ldots, n$, then also we get the same spectrum as in Theorem \ref{spec-cor}. Thus, we can construct infinitely many pairs of non-isomorphic graphs with same  $\epsilon$-spectra.
\end{rmk}

The following result is known about the $\epsilon$-spectrum of complete product of a graph $G$ with the complete graph $K_1$.
\begin{thm}\label{lemma-main-ecc1}\cite[Lemma 3.6]{ecc-main}
	Let $G$ be a $r$-regular graph with diameter $2$, and let $r, \lambda_2, \ldots, \lambda_n$ be the eigenvalues of the adjacency matrix of $G$, then
	\[
	\epsilon(G \vee K_1)=
	\left[ {\begin{array}{cc}
		2A(\overline{G})& J_{n \times 1} \\
		J_{1 \times n}& 0 \\
		\end{array} } \right],
	\] and \[
	spec_{\epsilon}(G \vee K_1)=
	\left\{ {\begin{array}{cccc}
		(n-r-1)\pm \sqrt{(n-r-1)^2 + n} & -2(\lambda_2+1)&\hdots& -2(\lambda_n+1) \\
		1 &1 &1&1\\
		\end{array} } \right\}.
	\]
\end{thm}

Using the above theorem, next we shall compute the eigenvalues of the wheel graph.
\begin{thm}
	Let $W_{n+1}$ be the wheel graph on $n+1$ vertices, with $n\geq 4$. Then, \[
	spec_{\epsilon}(W_{n+1})=
	\left\{ {\begin{array}{cccc}
		(n-3)\pm \sqrt{(n-3)^2+n} & -2(\lambda_2+1)&\hdots& -2(\lambda_n+1) \\
		1 &1 &1&1\\
		\end{array} } \right\}.
	\]
\end{thm}

\begin{proof}
	Let $C_n$ be a cycle of length $n(n\geq 4)$. Since the cycle is a $2$-regular graph, $2$ is an eigenvalue of the adjacency matrix of  $C_n$. Let the eigenvalues of adjacency matrix of  $C_n$ be  $2, \lambda_2,\hdots, \lambda_n$.
	
	It is easy to see that
	\[
	\epsilon(W_{n+1})=
	\left[ {\begin{array}{cc}
		2A(\overline{C_n})& J_{n\times 1} \\
		J_{1\times n} & 0 \\
		\end{array} } \right].
	\]
	Now the result follows from  Theorem \ref{lemma-main-ecc1}.
\end{proof}

In the next theorem, we compute the $\epsilon$-spectra of the barbell graphs.
\begin{thm}
	Let $B_{n,n}$ be the  $n$-barbell graph. Then,  \[
	spec_{\epsilon}(B_{n,n})=
	\left\{ {\begin{array}{ccc}
		0 & \frac{3(n-1)\pm \sqrt {9n^2-2n-7}}{2} &-{\frac{3(n-1)\pm \sqrt {9n^2-2n-7}}{2} } \\
		2(n-2) &1&1\\
		\end{array} } \right\}.
	\]
\end{thm}
\begin{proof} Let $K_n$ be the complete graph on $n$ vertices with vertex set $\{v_1,v_2,\hdots,v_n\}$, and let us consider a copy of $K_n$ with vertex set $\{w_1, w_2, \hdots, w_n\}$. Let $B_{n,n}$ be the barbell graph obtained by joining the vertices of $v_1$ and $w_1$ in the two copies of $K_n$. Then, the eccentric matrix of $B_{n,n}$ is given by
	\[
	\epsilon(B_{n,n})=
	\left[ {\begin{array}{cc}
		0_{n\times n}& A_{n\times n} \\
		{A}_{n\times n}& 0_{n\times n} \\
		\end{array} } \right],
	\]
	where \[
	A=
	\left[ {\begin{array}{cc}
		0& 2J_{1\times n-1} \\
		2J_{n-1 \times 1}& 3J_{n-1} \\
		\end{array} } \right].
	\]
	From Lemma $\ref{lem:A}$,  we get
	\[
	\sigma(A)=
	\left\{ {\begin{array}{cc}
		0 & \frac{3(n-1)\pm \sqrt {9n^2-2n-7}}{2}  \\
		(n-2) &1\\
		\end{array} } \right\}.
	\]
	Now, by Lemma \ref{b0-b1-lemma},  the spectrum  of the matrix $\epsilon(B_{n,n})$ is the union of eigenvalues of $A$ and $-A$. That is,
	\[
	spec_{\epsilon}(B_{n,n})=
	\left\{ {\begin{array}{ccc}
		0 & \frac{3(n-1)\pm \sqrt {9n^2-2n-7}}{2} &-{\frac{3(n-1)\pm \sqrt {9n^2-2n-7}}{2} } \\
		2(n-2) &1&1\\
		\end{array} } \right\}.
	\]
\end{proof}

In the next theorem, we compute the $\epsilon$-spectrum of the cocktail-party graph.

\begin{thm}
	Let $CP(n)$ be the cocktail-party graph on $2n$ vertices. Then,  \[
	\textbf spec_{\epsilon}(CP(n))=
	\left\{ {\begin{array}{cc}
		2 & -2  \\
		n & n\\
		\end{array} } \right\}.
	\]
\end{thm}
\begin{proof} Let $K_{2n}$ be the complete graph on $2n$ vertices. Let us delete the $n$ disjoint edges from the complete graph $K_{2n}$ to obtain  $CP(n)$ as follows: First  label $K_{2n}$ using the indices $\{1,2,\hdots,2n\}$ in clockwise direction, and then delete the edges $e_{ij}$, for $i=1,\hdots,n;j=n+1,n+2,\hdots,2n$, only when $i\equiv j(\mod ~n)$. Then, the eccentricity matrix of $CP(n)$ is given by
	
	\[
	\epsilon \big(CP(n)\big)=
	\left[ {\begin{array}{cc}
		0_{n\times n}& 2I_{n\times n} \\
		2I_{n\times n}& 0_{n\times n} \\
		\end{array} } \right].
	\]
	Therefore, by Lemma \ref{b0-b1-lemma}, we have
	$$  \textbf spec_{\epsilon}(CP(n))=
	\left\{ {\begin{array}{cc}
		2 & -2  \\
		n & n\\
		\end{array} } \right\}.$$
\end{proof}

%\begin{rmk}
%(i) If $G_1$ is complete, then \[
%   \epsilon(G_1 \vee G_2)=
%  \left[ {\begin{array}{cc}
%   A({G_1})& J\\
%   J& 2A(\overline{G_2}) \\
%  \end{array} } \right]
%\]
%(ii)  If $G_2$ is complete, then \[
%   \epsilon(G_1 \vee G_2)=
%  \left[ {\begin{array}{cc}
%   2A(\overline{G_1})& J\\
%   J& A({G_2}) \\
%  \end{array} } \right]
%\]
%(iii) If both $G_1$ and $G_2$ are complete, then $\epsilon(G_1\vee G_2)=J-I$.
%
%\textbf{So, in all these cases, we can compute the spectrum easily, let us mention  their spectrum.}
%\end{rmk}
\begin{thm}
	Let  $K_{n_1,\hdots,n_k}$ be  the complete $k$-partite graph such that $\sum_{i=1}^k n_i=n; n_i\geq 1$ and $k\leq n-1$. Then,
	\[
	spec_{\epsilon}(K_{n_1,\hdots,n_k})=
	\left\{ {\begin{array}{ccccc}
		(-2)^n & 2(n_1-1) & 2(n_2-1)&\hdots & 2(n_k-1)\\
		n-k &1 &1&\hdots &1\\
		\end{array} } \right\}.
	\]
\end{thm}

\begin{proof} The eccentricity matrix of $K_{n_1,\hdots,n_k}$ is given by
	\[
	\epsilon(K_{n_1,\hdots,n_k})=
	\left[ {\begin{array}{cccc}
		2(J_{n_1}-I_{n_1}) & 0 & \hdots  & 0 \\
		0 & 2(J_{n_2}-I_{n_2}) & \hdots & 0 \\
		\vdots & \vdots & \ddots & \vdots \\
		0 & 0& \hdots & 2(J_{n_k}-I_{n_k}) \\
		\end{array} } \right].
	\]
	Therefore, the spectrum of $\epsilon(K_{n_1,\hdots,n_k})$ is the union of eigenvalues of $2(J_{n_1}-I_{n_1})$, $2(J_{n_2}-I_{n_2})$,$\hdots$, $2(J_{n_k}-I_{n_k})$.
	
\end{proof}

In \cite{ecc-main}, the authors established that if $G$ is a $r$-regular graph with diameter $2$, then  \[
\epsilon(G \vee G)=
\left[ {\begin{array}{cc}
	2A(\overline{G})& 0 \\
	0& 2A(\overline{G}) \\
	\end{array} } \right].
\]
In the next theorem, we consider complete product of two non-complete graphs.  If the eigenvalues of adjacency matrices of $\overline{G_1}$ and $\overline{G_2}$ are known, then the $\epsilon$-spectrum of $G_1 \vee G_2$ is the union of the spectra  of $A(\overline{G_1})$ and $A(\overline{G_2})$.
\begin{thm}
	If $G_1$ and $G_2$ be any two non-complete connected graph, then
	\[
	\epsilon(G_1 \vee G_2)=
	\left[ {\begin{array}{cc}
		2A(\overline{G_1})& 0 \\
		0& 2A(\overline{G_2}) \\
		\end{array} } \right].
	\]
\end{thm}

\begin{proof}
	Easy to verify.
\end{proof}

\section{Inertia of eccentricity matrices of some classes of graphs}\label{sec5}
In this section, we compute the $\epsilon$-inertia of lollipop graph and path graph.  To begin with we compute the inertia of lollipop graph $L_{m,n}(n \geq 2)$.
\begin{thm}
	The inertia of the eccentricity matrix of  lollipop graph $L_{m,n}$ $(n \geq 2) $ is $(2,2,m+n-4)$.
\end{thm}
\begin{proof} Let $K_m$ be the complete graph on $m$ vertices with vertex set $\{1,\hdots,m\}$, and let $P_n$ be the path on $n$ vertices with vertex set $\{m+1,\hdots,m+n\}$. Form the $(m,n)$- lollipop graph by joining vertex $m$ of $K_m$ with the vertex $m+1$ of $P_n$.
	
	\textbf{Case(I):} Let $n$ be an even integer, say $n=2k$. Then, the eccentricity matrix of $L_{m,n}$ is given by
	\[
	\epsilon(L_{m,n})=
	\left[ {\begin{array}{cc}
		0_{m\times m}& A_{m\times 2k} \\
		{A^t}_{2k \times m} & 0_{2k \times 2k} \\
		\end{array} } \right],
	\]
	where, $$A=
	\begin{cases}
	\text{$j+1$} & \quad\text{if $j=k,k+1,\hdots,2k~\mbox{and}~i=1,2,\hdots,m-1$,}\\
	\text{$j$}& \quad\text{if $j=2k~\mbox{and}~i=m$,}\\
	\text{0} & \quad\text{otherwise.}\\
	\end{cases}$$
	It is easy to see that, $\epsilon(L_{m,n})$  has rank $4$.
	
	\textbf{Case(II)} Let $n$ be an odd integer, say $n=2k+1$. Then, the eccentricity matrix of $L_{m,n}$ is  given by
	\[
	\epsilon(L_{m,n})=
	\left[ {\begin{array}{cc}
		0_{m\times m}& B_{m\times 2k+1} \\
		{B^t}_{2k+1 \times m} & 0_{2k+1 \times 2k+1} \\
		\end{array} } \right],
	\]
	where, $$B=
	\begin{cases}
	\text{$j+1$} & \quad\text{if $j=k,k+1,\hdots,2k+1 ~\mbox{and}~i=1,2,\hdots,m-1$,}\\
	\text{$j$}& \quad\text{if $j=2k+1~\mbox{and}~i=m$,}\\
	\text{0} & \quad\text{otherwise.}\\
	\end{cases}$$
	
	It is clear that, the rank of $\epsilon(L_{m,n})$ is $4$.
	
	Therefore, in both cases, the multiplicity of  zero as an eigenvalue of  $\epsilon(L_{m,n})$ is  $m+n-4$.
	
	Let $C$ denote the $(m+n-1)\times (m+n-1)$ principal submatrix of $\epsilon(L_{m,n})$ obtained by deleting the  $(m+n)^{th}$ row and the $(m+n)^{th}$ column. Therefore, $C=
	\left[ {\begin{array}{cc}
		0& D \\
		D^t & 0 \\
		\end{array} } \right]$, where $D$ is an $m\times n-1$ matrix with rank $1$,  and therefore $rank\big(C\big)=2$. Since $tr\big(C\big)=0$, it has one positive eigenvalue and one negative eigenvalue.
	Since, $tr(\epsilon(L_{m,n}))=0$,$\epsilon(L_{m,n})$ has at least one  positive and at least one negative eigenvalue. We claim  that $\epsilon(L_{m,n})$ has exactly two positive eigenvalues and  two negative eigenvalues.
	
	Suppose not. Then, without loss generality,  the matrix  $\epsilon(L_{m,n})$  has $3$ negative eigenvalues and $1$ positive eigenvalue. Then, by interlacing theorem, the $(m+n-1)\times (m+n-1)$ principal submatrix $C$ has two negative eigenvalues and one positive eigenvalue, which is not true.
	Similarly, if $\epsilon(L_{m,n})$ has $3$ positive eigenvalues and $1$ negative eigenvalue, then,  again by interlacing theorem, the $(m+n-1)\times (m+n-1)$ principal submatrix $B$ has two negative eigenvalues and one positive eigenvalue. Hence, the inertia of $\epsilon(L_{m,n})$ $(n \geq 2)$ is $(2,2,m+n-4)$.
\end{proof}

In the next theorem, we compute the inertia of the path graph.
\begin{thm}
	Let $P_n$ be the path on $n$ vertices. Then, the inertia of eccentricity matrix  $\epsilon(P_n)$ is $(2,2,n-4)$.
\end{thm}
\begin{proof}
	Let $P_n$ be the path on $n$ vertices with vertex set $\{1,2,\ldots,n\}$, and let $\epsilon(P_n)$ be its eccentricity matrix. Then
	$$\epsilon(P_n)_{ij}=
	\begin{cases}
	\text{$j-1$} & \quad\text{if $i=1 ~\mbox{and}~ j=[\frac{n}{2}]+1,[\frac{n}{2}]+2,\hdots,n$,}\\
	\text{$j-i$}& \quad\text{if $j=n  ~\mbox{and}~ i=2,3,\hdots,[\frac{n}{2}]$,}\\
	\text{0} & \quad\text{otherwise,}\\
	\end{cases}$$
	where $[\frac{n}{2}]$ denotes the greatest integer not greater than $\frac{n}{2}$. Clearly, the rank of the eccentricity matrix $\epsilon(P_n)$ is $4$, and hence the algebraic multiplicity of $0$, as an eigenvalue of $\epsilon(P_n)$, is $n-4$. Now,  consider the $3\times 3$ principal submatrix $B$ of $\epsilon(P_n)$ obtained by taking $1^{st}$, $(n-1)^{th}$, and $n^{th}$ rows and columns of $\epsilon(P_n)$. Then
	\[ B=
	\left[ {\begin{array}{ccc}
		0 & n-2 &n-1 \\
		n-2 & 0&0 \\
		n-1 &0 &0\\
		\end{array} } \right],
	\]
	and the eigenvalues of $B$ are $- \sqrt{(n-1)^2+(n-2)^2},0, \sqrt{(n-1)^2+(n-2)^2}$. Therefore, by interlacing theorem,  $\epsilon(P_n)$ has two positive and two negative eigenvalues. Thus, the inertia of $\epsilon(P_n)$ is $(2,2,n-4)$.
\end{proof}

\textbf{Acknowledgement:} Iswar Mahato and  M. Rajesh Kannan would like to thank  Department of Science and Technology, India, for financial support through the Early Carrier Research Award (ECR/2017/000643).

\bibliographystyle{plain}
\bibliography{ecc-bib}

\end{document}